\documentclass{article}

\usepackage{pictexwd, dcpic}
\usepackage{amsfonts}
\usepackage{amsmath, amscd, amsxtra}
\usepackage{amssymb}
\usepackage{amsthm}
\usepackage{enumerate}
\usepackage{psfrag}
\usepackage{txfonts}
\usepackage{slashed}

\usepackage[usenames, dvipsnames]{color}
\usepackage{caption}

\newtheorem{prop}{Proposition}[section]
\newtheorem{theorem}{Theorem}[section]
\newtheorem{definition}{Definition}[section]
\newtheorem{corollary}{Corollary}[section]

\newtheorem{example}{Example}[section]
\newtheorem{lemma}{Lemma}[section]
\newtheorem*{remark}{Remark}
\title{Quandles of cyclic type with several fixed points}
\author{Ant\'onio Lages$\sp{1}$, Pedro Lopes$\sp{1, 2}$\\
$\sp{1}$Department of Mathematics\\
$\sp{2}$Center for Mathematical Analysis, Geometry and Dynamical Systems\\
Instituto Superior T\'ecnico\\
University of Lisbon\\
Av. Rovisco Pais\\
1049-001 Lisbon\\
Portugal\\
        \texttt{antonio.lages@tecnico.ulisboa.pt   pelopes@math.tecnico.ulisboa.pt}\\}

\begin{document}

\maketitle

\begin{abstract}
A quandle of cyclic type of order $n$ with $f\geq 2$ fixed points is such that each of its permutations splits into $f$ cycles of length $1$ and one cycle of length $n-f$. In this article we prove that there is only one such connected quandle, up to isomorphism. This is a quandle of order $6$ and $2$ fixed points, known in the literature as octahedron quandle. We prove also that, for each $f\geq 2$, the non-connected versions of these quandles only occur for orders $n$ in the range $f+2 \leq n \leq 2f$ and that, for each $f>1$, there is only one such quandle of order $2f$ with $f$ fixed points, up to isomorphism.  Still in the range $f+2 \leq n \leq 2f$, we present sufficient conditions for the existence of such quandles, writing down their permutations; we also show how to obtain new quandles form old ones, leaning on the notion of common fixed point.  \end{abstract}

\bigbreak



\section{Introduction}

\subsection{Quandles.}

The algebraic structure known as quandle appeared first in the literature in 1982, due to Joyce \cite{Joyce} and Matveev \cite{Matveev}, independently (see also \cite{Dehornoy} and \cite{FennRourke}). It was designed to constitute the algebraic counterpart of the Reidemeister moves \cite{Kauffman}. As such it turned out to be  an important tool in telling knots apart \cite{DionisioLopes}, \cite{BojarczukLopes}, \cite{carteretal}. Algebraists also find it interesting in the domain of Hopf algebras \cite{AndruskiewitschGrana}. It thus seem relevant to study the structure of quandles. In the current article we take another step in this direction by investigating and almost fully classifying a family of quandles. In \cite{LopesRoseman} quandles are regarded as sequences of permutations and based on the features of permutations conclusions are drawn. In particular, in \cite{LopesRoseman} quandles of the following sort are looked into. Given a positive integer $n$ we consider a quandle of order $n$, each of whose permutations split into a cycle of length $n-1$ and a fixed point; this fixed point complies with one of the quandles axioms - this is all detailed below. These quandles were subsequently called ``quandles of cyclic type''. They were also studied in \cite{KamadaTamaruWada}  and \cite{Vendramin}. In this article we work in the spirit of \cite{LopesRoseman} i.e., quandles as sequences of permutations, and we look into the classification of a generalization of ``quandles of cyclic type'' which we call ``quandles of cyclic type with several fixed points''. Details are supplied below in the text.

\subsection{Basic Definitions and Results}

The algebraic structure known as $\emph{quandle}$, introduced independently in \cite{Joyce} and \cite{Matveev}, is defined as follows.\par

\begin{definition}\label{def:quandle}
Let $Q$ be a set equipped with a binary operation denoted by $*$. The pair $(Q,*)$ is said to be a $\emph{quandle}$ if, for each $a, b, c\in Q$,
\begin{enumerate}
\item $a*a=a$ (idempotency);
\item $\exists ! x\in Q: x*b=a$ (right-invertibility);
\item $(a*b)*c=(a*c)*(b*c)$ (self-distributivity).
\end{enumerate}
\end{definition}

We present a few examples of quandles.

\begin{example}\textcolor{white}{.}

\begin{itemize}
    \item Let $G$ be a group and let $*$ be the binary operation on $G$ given by $a*b=bab^{-1}$, for every $a,b\in G$, where the juxtaposition on the right-hand side denotes group multiplication. Then, the pair $(G,*)$ is a quandle;
    \item For each $n\geq 2$, $(R_n,*)$ denotes the quandle whose underlying set is $\mathbb{Z}/n\mathbb{Z}$ and whose operation is $a*b=2b-a \mod n$. This is called the dihedral quandle (of order $n$);
    \item For each $n\geq 1$, $(T_n,*)$ denotes the quandle whose underlying set is $\{1,\dots, n\}$ and whose operation is $i*j=i$, $\forall i,j\in \{1,\dots, n\}$. This is called the trivial quandle (of order $n$);
    \item $Q_6^2$ is the quandle whose multiplication table is displayed in Table \ref{table:1}.

\end{itemize}
\end{example}

\begin{center}
\begin{tabular}{|c|c c c c c c|}
 \hline
 $\ast$ & 1 & 2 & 3 & 4 & 5 & 6\\
 \hline
 1 & 1 & 5 & 1 & 6 & 4 & 2\\
 2 & 6 & 2 & 5 & 2 & 1 & 3\\
 3 & 3 & 6 & 3 & 5 & 2 & 4\\
 4 & 5 & 4 & 6 & 4 & 3 & 1\\
 5 & 2 & 3 & 4 & 1 & 5 & 5\\
 6 & 4 & 1 & 2 & 3 & 6 & 6\\
 \hline
\end{tabular}
\captionof{table}{$Q_6^2$ multiplication table.}\label{table:1}
\end{center}

An alternative description of the structure of a quandle is the one given in the following theorem (\cite{Brieskorn}, \cite{FennRourke}). \par

\begin{theorem}\label{thm:equivdef}
Let $Q=\{1, 2, \dots, n\}$. Suppose a permutation $\mu_i$ from $S_n$, the symmetric group on $Q$, is assigned to each $i\in Q$. Then, the expression $j\ast i:=\mu_i(j), \forall  j \in Q$, yields a quandle structure if and only if $\mu_{\mu_i(j)}= \mu_i\mu_j\mu_i^{-1}$ and $\mu_i(i)=i$, $\forall i, j\in Q$. This quandle structure is uniquely determined by the set of $n$ permutations, $\{ \mu_1, \dots , \mu_n   \}$.
\end{theorem}

\begin{proof}
The proof can be found in \cite{Brieskorn}. For the sake of completeness, we repeat it here.\par
Suppose $(Q,*)$ is a quandle. Then, $\forall i\in Q$, we have that $i*i=i$, which is equivalent to saying that $\mu_i(i)=i, \forall i\in Q$. Moreover, given indices $i,j,k\in Q$, $(i*j)*k=(i*k)*(j*k)$ is equivalent to $\mu_k(\mu_j(i))=\mu_{\mu_k(j)}(\mu_k(i))$, which is equivalent to saying, by the right-invertibility axiom in the definition of quandle, that $\mu_{\mu_k(j)}=\mu_k\mu_j\mu_k^{-1}$ , and hence the result follows in one direction. Now we prove the converse. Indeed, the three axioms in Definition \ref{def:quandle} are satisfied. Both idempotency and self-distributivity are clear from the previous calculations. To check the right-invertibility, let $i,j\in Q$. Since $\mu_j$ is a bijection, there exists an unique $x\in Q$ such that $i=\mu_j(x)=:x*j$, and so the result follows.
\end{proof}

In this article, we study the properties of each quandle by analyzing the structure uniquely determined by its set of $n$ permutations. We also only address finite quandles. If such a quandle has order $n$, we take the underlying set to be $\{  1, 2, \dots , n \}$, without loss of generality.

\begin{definition}\label{def:quandles-as-perms}
Given a quandle $(Q,*)$, its permutations are the $\mu_i$'s referred to in the statement of Theorem \ref{thm:equivdef}, $\forall i\in Q$. Unless otherwise stated in the sequel, a $\mu_i$ always refers to a permutation from the quandle under discussion. We also write $(Q, \mu)$ to denote the same quandle from the point-of-view of permutations.
\end{definition}

\begin{definition}\label{def:quandle-iso}
Let $(Q, \ast)$ (respect., $(Q, \mu)$) and $(Q', \ast')$ (respect., $(Q', \mu')$) be two quandles. A bijection $\alpha : Q \longrightarrow Q'$ is a quandle isomorphism between these two quandles if, by definition, for any $i, j \in Q$, $\alpha (i\ast j) = \alpha(i) \ast' \alpha(j)$.
\end{definition}

\begin{prop} Keeping the notation and terminology of Defintion \ref{def:quandles-as-perms}, $\alpha$ is a quandle isomorphism if and only if, for any $i\in Q$, $\mu'_{\alpha (i)} = \alpha \mu_i \alpha^{-1}$.
\end{prop}
\begin{proof}
For any $i, j\in Q$, $$\alpha (j\ast i) = \alpha (j) \ast' \alpha (i) \Longleftrightarrow \alpha (\mu_i (j)) = \mu'_{\alpha (i)}(\alpha (j))  \Longleftrightarrow \alpha \, \mu_i  = \mu'_{\alpha (i)}\, \alpha  \Longleftrightarrow \mu'_{\alpha (i)} = \alpha \, \mu_i \, \alpha^{-1}   $$
\end{proof}

\subsection{Definition of Quandle of Cyclic Type with Several Fixed Points}

As stated in Theorem \ref{thm:equivdef}, a quandle of order $n$ is uniquely determined by a set of $n$ permutations, where each one of these permutations can be decomposed into a set of disjoint cycles. The lengths of these cycles define the $\emph{pattern}$ of each permutation.

\begin{definition}
The $\emph{pattern}$ of a permutation is the list of the lengths of the disjoint cycles making up the permutation.
\end{definition}

We can collect the information relative to the patterns of the $n$ permutations defining a quandle of order $n$ in order to define its $\emph{profile}$.

\begin{definition}
The $\emph{profile}$ of a quandle of order $n$ is the list of the patterns of the $n$ permutations defining the quandle.
\end{definition}

We now introduce the notion of $\emph{connected quandle}$ in order to state an important proposition.

\begin{definition}
A finite quandle $(Q,*)$ is said to be $\emph{connected}$ if, $\,\forall \,i,j\in Q,\,\exists\,k_1, k_2, \dots, k_n\in Q:$ \[j=(\cdots((i*k_1)*k_2)*\cdots*k_n)=\mu_{k_n}\circ\cdots\circ\mu_{k_2}\circ\mu_{k_1}(i).\]
\end{definition}

\begin{prop}\label{prop:connec}
Connected finite quandles have constant profiles.
\end{prop}

\begin{proof}
Let $(Q,*)$ be a connected quandle. Then, given $i,j\in Q,\,\exists\,k_1, k_2, \dots, k_n\in Q:$ \[j=(\cdots((i*k_1)*k_2)*\cdots*k_n)=\mu_{k_n}\circ\cdots\circ\mu_{k_2}\circ\mu_{k_1}(i)\Rightarrow\mu_j=\mu_{k_n}\cdots\mu_{k_2}\mu_{k_1}\mu_i\mu_{k_1}^{-1}\mu_{k_2}^{-1}\cdots\mu_{k_n}^{-1},\] by Theorem \ref{thm:equivdef}, and since conjugate permutations have the same pattern, the result follows.
\end{proof}

\begin{remark}
Note that  quandles with constant profile do not have to be connected. For example, the trivial quandle of order $n$, $(T_n, *)$, has constant profile and it is not connected. The same for dihedral quandles of even order.
\end{remark}

Finally, we introduce the key notion of $\emph{quandles of cyclic type with several fixed points}$.

\begin{definition}
Given $n, f\in\mathbb{N}$, $n-2\geq f>1$, a $\emph{quandle of  cyclic type of order}$ $n$ $\emph{with}$ $f$ $\emph{fixed points}$  is a quandle of order $n$ with constant profile given by $$\bigg\{ \underbrace{\{\underbrace{1,\dots,1}_{f},n-f\},\dots,\{\underbrace{1,\dots,1}_{f},n-f\}}_n \bigg\} .$$ When there is no need to refer to its order or to its number of fixed points we refer to each of these quandles as $\emph{quandle of  cyclic type with several fixed points}$.
\end{definition}

The previous definition means that each one of the $n$ permutations defining a quandle of cyclic type of order $n$ with $f$ fixed points splits into the following types and numbers of disjoint cycles. One cycle of length $n-f(>1)$ and $f$ cycles of length $1$.\par

In passing, we note, by inspection of Table \ref{table:1}, that $Q_6^2$ is a connected quandle. Hence, by Proposition \ref{prop:connec}, it has a constant profile. But more than that, $Q_6^2$ is, in fact, a quandle of cyclic type of order $6$ with $2$ fixed points. In this article, we show that this is the only connected quandle of cyclic type with several fixed points.\par

In the sequel, we use the following notation.

\begin{definition}
Let $Q$ be a quandle of cyclic type with several fixed points of order $n$. As per Theorem 2.1.1, its $n$ permutations are denoted $\mu_i$, $i\in\{1, \dots, n\}$. In particular, $\mu_i(i) = i, \forall i\in\{1, \dots, n\}$. The set of fixed points of $\mu_i$ is denoted $F_i$, $i\in\{1, \dots, n\}$. The set of points in the non-singular cycle of $\mu_i$ is denoted $C_i$, $i\in\{1, \dots, n\}$. We note that $C_i\cap F_i=\emptyset$ and $C_i\cup F_i=\{1, \dots, n\}$, $\forall i\in\{1, \dots, n\}$.
\end{definition}

\subsection{Statement of the Results of this Article}\label{subsect:results}

\begin{theorem}\label{thm:main1}
Let $n$ and $f$ be positive integers such that $n > f + 1 > 2$, and assume $Q$ is a quandle of cyclic type of order $n$ with $f$ fixed points. Then, the following hold.
\begin{enumerate}
\item Assume $n > 2f$.
\begin{enumerate}
\item Then, $Q$ is connected.
\item  Moreover, there is only one such quandle, up to isomorphism. It occurs for  $n=6$ and $f=2$. This quandle  is  $Q_6^2$ (see Table \ref{table:1}, above).
\end{enumerate}
\item Assume $n \leq  2f$.
\begin{enumerate}
\item Then, $Q$ is not connected.
\item  If $n=2f$, there is only one such quandle, up to isomorphism. Its permutations are $$\mu_1 = \mu_2 = \cdots = \mu_f = \bigg( f+1 \quad f+2 \quad \cdots \quad 2f \bigg)$$ $$\mu_{f+1} = \mu_{f+2} = \cdots = \mu_{2f} = \bigg( 1 \quad 2 \quad \cdots \quad f \bigg) .$$
\end{enumerate}
\end{enumerate}
\end{theorem}

\begin{theorem}\label{thm:main2}
Let $n$ and $f$ be positive integers such that $n > f + 1 > 2$ and $n \leq 2f$ (corresponding to the non-connected case in Theorem \ref{thm:main1}). If $(n-f) \mid f$, there is a quandle of cyclic type of order $n$ with $f$ fixed points whose permutations are given as follows. For each $i$ such that $1 \leq i \leq \frac{n}{n-f}$,  $$ \mu_{(i-1)(n-f)+1}=\mu_{(i-1)(n-f)+2}= \cdots =\mu_{i(n-f)}=$$ $$= \bigg( i(n-f)+1\quad  i(n-f)+2 \quad \cdots \quad (i+1)(n-f)\bigg) ,$$ with indices read mod  $\frac{n}{n-f}$.
\end{theorem}

\begin{theorem}\label{thm:extracting-common-fixed-point}
Suppose $f$ is an integer strictly greater than $2$ and $n$ a positive integer such that $f+2 \leq n \leq 2f$. Consider a quandle of cyclic type of order $n$ and $f$ fixed points over the set $Q = \{ 1, 2, \dots , n  \}$ with sequence of permutations $\mu_i$ with $i\in \{  1, 2, \dots , n  \}$. Assume further there is a $g_0 \in Q$ such that $\mu_i(g_0) = g_0$, for any $i\in Q$. Then, the set $Q' = Q \setminus \{  g_0 \}$ along with the sequence of permutations $\mu'_i = \mu_i|_{Q'}$ for each $i\in Q'$ defines a quandle of cyclic type of order $n-1$ with $f-1$ fixed points. We call this the extraction of the common fixed point $g_0$.
\end{theorem}

\begin{theorem}\label{thm:adjoining-common-fixed-point}
Let $n$ be an integer greater than $2$. Let $Q$ be the underlying set of a quandle whose permutations are denoted $\mu_i$, for each $i\in Q$. Let $g_0\notin Q$ and consider the set $Q' = Q \cup \{  g_0 \}$. Suppose there is a permutation, $\mu$, of the elements of $Q$, such that $\mu\mu_i = \mu_i\mu$, for each $i\in Q$. Then, $Q'$ along with the permutations $$\mu'_i = (g_0)\mu_i \qquad \text{ for each } i\in Q \qquad \qquad \text{ and } \qquad \qquad \mu'_{g_0} = (g_0)\mu$$ is a quandle with a common fixed point, $g_0$. We call this the adjoining of a common fixed point $g_0$.
\end{theorem}

\begin{corollary}\label{cor:an-infinite-sequence-adjoining-common-fixed-points}
Let $(Q, \mu)$ be a quandle of cyclic type of order $n$ and $f$ fixed points with $(n-f)\, |\, f$ as in Theorem \ref{thm:main2}. Then any two permutations are either equal or move points from disjoint sets. So adjoining a common fixed point $g_0$ is accomplished by taking $\mu'_{g_0}=(g_0)\mu_{i_0}$ by picking any $i_0 \in Q$ and for any $j\in Q$, $\mu'_j = (g_0)\mu_j$.

Furthermore, this procedure can be iterated indefinitely, giving rise to an infinite sequence of quandles $Q_k$ of cyclic type of order $n+k$ and $f+k$ fixed points such that $f+k+2 \leq n + k \leq 2(f + k)$.
\end{corollary}

\subsection{Organization}\label{subsect:org}

The Sections below are devoted to the proofs of these facts. In Section \ref{sect:cyclic-type-1st-properties} we prove that quandles of cyclic type of order $n$ with $f$ fixed points in the range $n > 2f$ are connected (assertion  $1.(a)$ in Theorem \ref{thm:main1}) and that for $n=2f$ there is  only one such quandle, up to isomorphism  and that this quandle is not connected (assertion $2.(b)$ and assertion $2.(a)$ (for $n=2f$) in Theorem \ref{thm:main1}).  In Subsection \ref{subsect:disconnected-range-f+2-n-2f} we prove that quandles of cyclic type of order $n$ and $f$ fixed points in the range $f+2 \leq n \leq 2f$ are not connected; this is the $2.(a)$  part (for $n<2f$) in Theorem \ref{thm:main1}. In Section \ref{sect:n>2f} we prove that, up to isomorphism, there is only one quandle of cyclic type of order $n$ with $f$ fixed points in the range $n > 2f$. This quandle occurs for $n=6$ and $f=2$ and is known as the octahedron quandle (assertion $1.(b)$ in Theorem \ref{thm:main1}). This completes the proof of Theorem \ref{thm:main1}.  In Section \ref{sect:families-range-f+2-n-2f} we prove Theorems \ref{thm:main2}, \ref{thm:extracting-common-fixed-point}, and \ref{thm:adjoining-common-fixed-point}. 
Finally, in Section \ref{sect:further-research} we collect a few questions for further research.

\subsection{Acknowledgements.}\label{subsect:ack}

P.L. acknowledges support from FCT (Funda\c c\~ao para a Ci\^encia e a Tecnologia), Portugal, through project FCT PTDC/MAT-PUR/31089/2017, ``Higher Structures and Applications''.

\section{Quandles of Cyclic Type with Several Fixed Points - First Properties and Examples.}\label{sect:cyclic-type-1st-properties}

In this Section, we state and prove a theorem about the structure of quandles of cyclic type with several fixed points. This theorem provides a number of conditions quandles of cyclic type of order $n$ with $f$ fixed points such that $n\geq 2f$ must verify. This theorem is key to proceed to a classification of quandles of cyclic type.

\subsection{Associate Indices}

Quandles of cyclic type of order $n\geq2f$ have a very useful property, which is a consequence of the following proposition.

\begin{prop}\label{prop:assocind}
Let $Q$ be a quandle of cyclic type of order $n$ with $f$ fixed points such that $n\geq2f$, and let $F_k=\{k, g_k^1, \dots, g_k^{f-1}\}$ be the set of $f$ fixed points of $\mu_k$. Then $\mu_{g_k^i}(g)=g$, $\forall g\in F_k$, $\forall i\in \{1, \dots, f-1\}$.
\end{prop}

\begin{proof}
For each $i\in \{1, \dots, f-1\}$ and for each $g\in F_k\setminus \{g^i_k\}$, we have \[\mu_{g_k^i}(g)=\mu_{\mu_k(g_k^i)}(g)=\mu_k\mu_{g_k^i}\mu_k^{-1}(g)=\mu_k(\mu_{g_k^i}(g)),\] that is, $\mu_k$ fixes $\mu_{g_k^i}(g)$ and hence $\mu_{g_k^i}(g)\in F_k\setminus\{g_k^i\}$, as $\mu_{g_k^i}(g_k^i)=g_k^i$. Therefore, the restriction of $\mu_{g_k^i}$ to $F_k\setminus\{g_k^i\}$ is a bijection from this set to itself. Thus, we must have $\mu_{g_k^i}(g)=g$, $\forall g\in F_k\setminus\{g_k^i\}$. Otherwise, $\mu_{g_k^i}$ would have a cycle of length $1<l\leq f-1=|F_k\setminus\{g_k^i\}|$, that would also verify $l=n-f\geq2f-f=f$, which is a contradiction. Hence, the result follows.
\end{proof}

We introduce the notions of $\emph{associate indices}$ and $\emph{associate permutations}$.

\begin{definition}\label{def:assocind}
Let $Q$ be a quandle of order $n$ with permutations denoted by $\mu_k$, $k\in\{1, \dots, n\}$.

\begin{itemize}
    \item If $i$ and $j$ are different indices such that $\mu_i(j)=j$ and $\mu_j(i)=i$, we say that $i$ and $j$ are $\emph{associate indices}$;
    \item If $i$ and $j$ are associate indices then $\mu_i$ and $\mu_j$ are said $\emph{associate permutations}$;
\end{itemize}

\end{definition}

\begin{corollary}
With the terminology introduced in Definition \ref{def:assocind}, Proposition \ref{prop:assocind} states that, for each quandle of cyclic type whose order $n$ and number of fixed points $f$ satisfies $n\geq 2f$, associate permutations have the same sets of fixed points.
\end{corollary}

In the sequel, we assume the order $n$ of any quandle of cyclic type to be greater than or equal to $2f$, unless otherwise stated. Therefore, Proposition \ref{prop:assocind} always applies.\par
We now prove the main result of this section, which is a consequence of the previous results.

\begin{corollary}\label{cor:order}
Assume $n$ is the order of a quandle of cyclic type with $f$ fixed points. If $n\geq2f$, then $n$ is a multiple of $f$ i.e., $n=cf$ for some integer $c\geq 2$.
\end{corollary}

\begin{proof}
We show that if there is an index $s$ such that $\mu_s(k)=k$, for some $k$, then $s\in F_k$. For each $g\in F_s\setminus\{k\}$, \[\mu_k(g)=\mu_{\mu_s(k)}(g)=\mu_s\mu_k\mu_s^{-1}(g)=\mu_s(\mu_k(g)),\] that is, $\mu_s$ fixes $\mu_k(g)$ and hence $\mu_k(g)\in F_s\setminus\{k\}$, as $\mu_k(k)=k$. Therefore, the restriction of $\mu_k$ to $F_s\setminus\{k\}$ is a bijection from this set to itself. Arguing as in the proof of Proposition \ref{prop:assocind}, we must have $\mu_k(g)=g$, $\forall g\in F_s\setminus\{k\}$, which implies that $F_s=F_k$, and in particular, $s\in F_k$. Thus, the sets of fixed points corresponding to two permutations are either equal or disjoint. Therefore, the order $n$ of a quandle of cyclic type with $f$ fixed points, with $n\geq2f$, has to be a multiple of $f$.
\end{proof}

The following proposition is now an immediate consequence of our previous considerations.

\begin{prop}\label{prop:equivrel}
Given a quandle of cyclic type of order $n$ with $f$ fixed points and $n\geq 2f>2$, ``i is associate to j'' generates an equivalence relation on the underlying set of the quandle.
\end{prop}

\begin{proof}
The equivalence relation is ``$i$ is associate to $j$ or $i=j$''. Since $i$ is a fixed point of $\mu_i$, then $i\in F_i$. Moreover, any two sets $F_i$ and $F_j$ are either equal or disjoint, thus the result follows.
\end{proof}

\begin{example}

$R_4$, the dihedral quandle of order $4$, whose multiplication table is displayed in Table \ref{table:3}, is a quandle of cyclic type of order $4$ with $2$ fixed points.

\begin{center}
\begin{tabular}{|c|c c c c|}
 \hline
 \emph{$\ast$} & \emph{1} & \emph{2} & \emph{3} & \emph{4}\\
 \hline
 \emph{1} & \emph{1} & \emph{3} & \emph{1} & \emph{3}\\
 \emph{2} & \emph{4} & \emph{2} & \emph{4} & \emph{2}\\
 \emph{3} & \emph{3} & \emph{1} & \emph{3} & \emph{1}\\
 \emph{4} & \emph{2} & \emph{4} & \emph{2} & \emph{4}\\
 \hline
\end{tabular}
\captionof{table}{$R_4$ multiplication table.}\label{table:3}
\end{center}

By Proposition \ref{prop:equivrel}, \text{``i is associate to j''} generates an equivalence relation on $R_4$, which is also a congruence relation on this set, as it respects the binary operation of the quandle. In Table \ref{table:4}, we see the quotient of $R_4$ by this congruence relation, which we denote by $\sim$. This quotient is clearly  isomorphic to $T_2$. In particular, $R_4$ is not simple since it admits a non-trivial quotient.

\begin{center}
\begin{tabular}{|c|c c|}
 \hline
 \emph{$\overline{\ast}$} & \emph{\{1,3\}} & \emph{\{2,4\}}\\
 \hline
 \emph{\{1,3\}} & \emph{\{1,3\}} & \emph{\{1,3\}}\\
 \emph{\{2,4\}} & \emph{\{2,4\}} & \emph{\{2,4\}}\\
 \hline
\end{tabular}
\captionof{table}{$R_4/\sim$ multiplication table.}\label{table:4}
\end{center}

\end{example}

\subsection{\boldmath Connected Quandles of Cyclic Type of Order $n$ with $f$ Fixed Points in the Range $n\geq 2f$ - First Properties and Example.}

From this point on, we are only working with connected quandles of cyclic type of order $n\geq 2f$. The two following propositions tell us whether a quandle of cyclic type of order $n\geq 2f$ is connected or not.

\begin{prop}\label{prop:n=2f-notconnected}
If $Q$ is a quandle of cyclic type of order $n$ with $f$ fixed points such that $n=2f$, $Q$ is not connected.
\end{prop}

\begin{proof}
Suppose $Q$ is a quandle of cyclic type of order $n$ with $f$ fixed points such that $n=2f$ and let $i,j\in Q$ be two non-associate indices. Then $F_i\cup F_j=Q$. Moreover, since $n-f=2f-f=f$, for any $a\in Q$, if $F_i$ (respect., $F_j$) is the set of fixed points of $\mu_a$, then $C_a = F_j$ (respect., $C_a = F_i$). Then, for each $a\in Q$, $\mu_a(F_i)=F_i$ and $\mu_a(F_j)=F_j$. Therefore, we conclude that $Q$ is not connected.
\end{proof}

\begin{prop}
Every quandle of cyclic type of order $n$ with $f$ fixed points such that $n>2f$ is connected.
\end{prop}

\begin{proof}
Suppose $Q$ is a quandle of cyclic type of order $n$ with $f$ fixed points such that $n=cf$, with $c\geq 3$ by Corollary \ref{cor:order}. Let $i,j\in Q$, with $i\neq j$. If $i$ and $j$ are associate indices, then for any $k\notin F_i$, $i, j \in C_k$. Hence, there exists an integer $a\in\{1, \dots, n-f-1\}$ such that $\mu_k^a(i)=j$. Now, assume $i$ and $j$ are not associate indices. Since there are at least three distinct sets of associate indices, there is at least one set $F_{k_0}$ such that both $i$ and $j$ do not belong to $F_{k_0}$. Therefore $i, j \in C_{k_0}$ and so there exists an integer $a\in\{1, \dots, n-f-1\}$ such that $\mu_{k_0}^a(i)=j$. We conclude that $Q$ is connected.
\end{proof}

\begin{remark}
In the sequel, we assume the order $n$ of any quandle to be greater than $2f$, unless otherwise stated. In this condition, all our quandles of cyclic type are connected.
\end{remark}

We now use some of the equalities $\mu_{\mu_i(j)}=\mu_i\mu_j\mu_i^{-1}$ the permutations defining these quandles have to verify to derive a number of conditions these quandles have to satisfy in order to be cyclic.\par

\begin{theorem}\label{thr:thethr}
Consider a quandle of cyclic type of order $n$ with $f$ fixed points such that $n>2f$. Modulo isomorphism, its sequence of permutations satisfies the following conditions.

\begin{enumerate}
\item $\mu_n = (1\,\,2\,\,3\,\cdots\,n-f)(n-f+1)\cdots(n-1)(n)$;
\item if $h$ and $h'$ are associate indices then $\mu_h=\mu_{h'}^{l_{h,h'}}$, where $GCD(n-f, l_{h,h'})=1$, $1\leq l_{h,h'}<n-f$;
\item $\mu_k=\mu_n^k\mu_{n-f}\mu_n^{-k}$, for all $1\leq k\leq n-f$;
\item $\mu_{n-f}\mu_a\mu_{n-f}^{-1}=\mu_n^{\mu_{n-f}(a)}\mu_{n-f}\mu_n^{-\mu_{n-f}(a)}$, $\forall a\in F_n$;
\item $\mu_{n-f}^{-1}\mu_a\mu_{n-f}=\mu_n^{\mu_{n-f}^{-1}(a)}\mu_{n-f}\mu_n^{-\mu_{n-f}^{-1}(a)}$, $\forall a\in F_n$;
\item $\forall m\in\{1,\dots, n-f\}\setminus\{\mu_{n-f}^{-1}(n-f+1), \dots,\mu_{n-f}^{-1}(n)\}$, there exists an integer $1\leq k_m < n-f$ such that $\mu_n^{-\mu_{n-f}(m)}\mu_{n-f}\mu_n^m=\sigma\tau^{k_m}$, where $\sigma$ is a permutation of $F_{n-f}$ and $\tau$ is the cycle of length $n-f$ in $\mu_{n-f}$.
\end{enumerate}

\end{theorem}

\begin{proof}
\textcolor{white}{.}\par
\textit{1.} We can assume that $\mu_n$ is given by $(1\,\,2\,\,3\,\cdots\,n-f)(n-f+1)\cdots(n-1)(n)$ without loss of generality. If necessary, we may relabel the indices. {\bf \boldmath This expression for $\mu_n$ will be assumed in the sequel} - except for Subsection \ref{subsect:disconnected-range-f+2-n-2f}.\\ \par
\textit{2.} Suppose $h$ and $h'$ are associate indices and let $F_h=\{h, h', g_h^1, \dots, g_h^{f-2}\}$. Therefore, we have that $\mu_h=(h_1\,\dots\,h_{n-f})(h)(h')(g_h^1)\cdots(g_h^{f-2})$ and $\mu_{h'}=(h'_1\,\dots\,h'_{n-f})(h)(h')(g_h^1)\cdots(g_h^{f-2})$, and hence \[\mu_{h'}=\mu_{\mu_h(h')}=\mu_h\mu_{h'}\mu_h^{-1}\Leftrightarrow \mu_{h'}\mu_h=\mu_h\mu_{h'}\Leftrightarrow\]\[\Leftrightarrow(h_1 \dots h_{n-f})(h'_1 \dots h'_{n-f})=(h'_1 \dots h'_{n-f})(h_1 \dots h_{n-f}),\] that is, the two cycles $(h_1\,\dots\,h_{n-f})$ and $(h'_1\,\dots\,h'_{n-f})$ commute in $S_{\{h_1,\dots,h_{n-f}\}}= S_{\{h'_1,\dots,h'_{n-f}\}}$. Thus, $\mu_h=\mu_{h'}^{l_{h,h'}}$ (see \cite{Rotman}, for instance), where $l_{h,h'}$ satisfies $GCD(n-f, l_{h,h'})=1$, $1\leq l_{h,h'}<n-f$ (otherwise $\mu_h$ would not have a cycle of length $n-f$).\\

\textit{3.} First, we note that $\mu_1=\mu_{\mu_n(n-f)}=\mu_n\mu_{n-f}\mu_n^{-1}$. If $\mu_k=\mu_n^k\mu_{n-f}\mu_n^{-k}$ then \[\mu_{k+1}=\mu_{\mu_n(k)}=\mu_n\mu_k\mu_n^{-1}=\mu_n\mu_n^k\mu_{n-f}\mu_n^{-k}\mu_n^{-1}=\mu_n^{k+1}\mu_{n-f}\mu_n^{-(k+1)},\] (where we read the free indices modulo $n-f$), whence we proved \textit{3.} by induction.\\

\textit{4.} Let $a\in F_n$ and assume $\mu_{n-f}(a)\notin F_n$. Let $i\notin F_n$. On one hand, we have, by assertions \textit{2.} and \textit{3.},
\begin{equation}\label{eq:4}
\mu_{\mu_i(a)}=\mu_i\mu_a\mu_i^{-1}=\mu_n^i\mu_{n-f}\mu_n^{-i}\mu_a\mu_n^i\mu_{n-f}^{-1}\mu_n^{-i}=\mu_n^i\mu_{n-f}\mu_a\mu_{n-f}^{-1}\mu_n^{-i}.
\end{equation}
On the other hand, again by assertion \textit{3.}, \[\mu_i(a)=\mu_n^{i}\mu_{n-f}\mu_n^{-i}(a)=\mu_n^{i}\mu_{n-f}(a)=\mu_n^{i}\mu_n^{\mu_{n-f}(a)}(n-f)=\mu_n^{i+\mu_{n-f}(a)}(n-f),\] which implies that
\begin{equation}\label{eq:5}
\mu_{\mu_i(a)}=\mu_{\mu_n^{i+\mu_{n-f}(a)}(n-f)}=\mu_n^i\mu_n^{\mu_{n-f}(a)}\mu_{n-f}\mu_n^{-\mu_{n-f}(a)}\mu_n^{-i}
\end{equation}
Combining \ref{eq:4} and \ref{eq:5}, we get \[\mu_{n-f}\mu_a\mu_{n-f}^{-1}=\mu_n^{\mu_{n-f}(a)}\mu_{n-f}\mu_n^{-\mu_{n-f}(a)}.\] We now prove the following lemma, which completes the proof of assertion \textit{4.}.

\begin{lemma}\label{lemma:ups}
Given $a\in F_n$, $\mu_{n-f}(a)\notin F_n$.
\end{lemma}

\begin{proof}
Suppose $\mu_{n-f}(a)\in F_n$. Then, for $1\leq k\leq n-f$, \[\mu_k(a)=\mu_n^k \mu_{n-f}\mu_n^{-k}(a)=\mu_n^k \mu_{n-f}(a)=\mu_{n-f}(a).\]
This would force the pairs of associate permutations from $\mu_1$ to $\mu_{n-f}$ to be equal to each other. In fact, by assertion \textit{2.}, if two associate permutations have the same image at a point belonging to their non-singular cycles, these permutations have to be the same. Suppose, now, that for a certain index $b\in F_n$, $\mu_{n-f}(b)\notin F_n$. Then, for $1\leq k\leq n-f$, \[\mu_k(b)=\mu_n^k \mu_{n-f}\mu_n^{-k}(b)=\mu_n^k \mu_{n-f}(b)=\mu_{n-f}(b)+k,\] which would force the pairs of permutations whose indices are associate from $\mu_1$ to $\mu_{n-f}$ to be all different from each other, which is a contradiction. Therefore, $\mu_{n-f}(F_n)=F_n$. But $\mu_{n-f}$ does not fix any element from $F_n$ since $F_n \cap F_{n-f}=\emptyset$. Then, this implies that $\mu_{n-f}$ has a cycle of length at most $f$, which is again a contradiction, since $n-f>f$. Hence, $\mu_{n-f}(a)\notin F_n$.
\end{proof}

\textit{5.} Let $a\in F_n$ and let $i\notin F_n$ be the index such that $\mu_{n-f}(i)=a$. Note that otherwise $i$ would take up the role of $a$ in Lemma \ref{lemma:ups} and $a$ would not belong to $F_n$. Then, \[\mu_a=\mu_{\mu_{n-f}(i)}=\mu_{n-f}\mu_i\mu_{n-f}^{-1}=\mu_{n-f}\mu_n^i\mu_{n-f}\mu_n^{-i}\mu_{n-f}^{-1}.\] Since $i=\mu_{n-f}^{-1}(a)$, we conclude that $\mu_{n-f}^{-1}\mu_a\mu_{n-f}=\mu_n^{\mu_{n-f}^{-1}(a)}\mu_{n-f}\mu_n^{-\mu_{n-f}^{-1}(a)}$.\\

\textit{6.} Let $m$ be any index belonging to the set $\{1,\dots,n-f\}$ such that the index $\mu_{n-f}(m)$ belongs to the set $\{1, \dots, n-f\}\setminus\{\mu_{n-f}(n), \dots, \mu_{n-f}(n-f+1)\}$. Then, by assertion \textit{3.},  \[\mu_{\mu_{n-f}(m)}=\mu_{n-f}\mu_m\mu_{n-f}^{-1}=\mu_{n-f}\mu_n^m\mu_{n-f}\mu_n^{-m}\mu_{n-f}^{-1}.\] Since we also have $\mu_{\mu_{n-f}(m)}=\mu_n^{\mu_{n-f}(m)}\mu_{n-f}\mu_n^{-\mu_{n-f}(m)}$, then  \[\mu_{n-f}\mu_n^m\mu_{n-f}\mu_n^{-m}\mu_{n-f}^{-1}=\mu_n^{\mu_{n-f}(m)}\mu_{n-f}\mu_n^{-\mu_{n-f}(m)},\] which is equivalent to having \[(\mu_n^{-\mu_{n-f}(m)}\mu_{n-f}\mu_n^m)\mu_{n-f}=\mu_{n-f}(\mu_n^{-\mu_{n-f}(m)}\mu_{n-f}\mu_n^m),\] that is, $\mu_n^{-\mu_{n-f}(m)}\mu_{n-f}\mu_n^m$ and $\mu_{n-f}$ commute in $S_n$. The number of elements in the centralizer of $\mu_{n-f}$ in $S_n$, $|C_{S_n}(\mu_{n-f})|$, is given by \[|C_{S_n}(\mu_{n-f})|=\frac{|S_n|}{|\mu_{n-f}^{S_n}|},\qquad \textrm{see \cite{Rotman}, for instance},\] where $|\mu_{n-f}^{S_n}|$ denotes the number of elements of $S_n$ with the same pattern as $\mu_{n-f}$. In fact, we have \[|\mu_{n-f}^{S_n}|=\frac{n(n-1)\dots(f+1)}{n-f}=\frac{n!}{(n-f)f!},\] and since $|S_n|=n!$, we conclude that $|C_{S_n}(\mu_{n-f})|=(n-f)f!$. However, we know exactly what these $(n-f)f!$ permutations are. Let $\tau$ be the cycle of length $n-f$ of $\mu_{n-f}$. Indeed, $\tau^k$ commutes with $\mu_{n-f}$, $\forall k\in\{1, \dots,n-f\}$. Moreover, any permutation of the $f$ fixed points of $\mu_{n-f}$ commutes with $\mu_{n-f}$. The former type of permutation $\tau^k$ only moves elements within $C_{n-f}$ whereas the latter type of permutation only moves elements within $F_{n-f}$. Composing permutations from these two commuting types of permutations, we get a total of $(n-f)f!$ permutations commuting with $\mu_{n-f}$, which is precisely the number of permutations we found before. Therefore, we may conclude that $\forall m\in\{1,\dots, n-f\}\setminus\{\mu_{n-f}^{-1}(n-f+1), \dots,\mu_{n-f}^{-1}(n)\}$, there exists an integer $1\leq k_m<n-f$ such that $\mu_n^{-\mu_{n-f}(m)}\mu_{n-f}\mu_n^m=\sigma\tau^{k_m}$, where $\sigma$ is a permutation of $F_{n-f}$ and $\tau$ is the cycle of length $n-f$ in $\mu_{n-f}$.

\end{proof}

\begin{corollary}\label{cor:thethr}
Assertions $1., 2., $ and $3.$ in Theorem \ref{thr:thethr} are still valid if $n=2f$.
\end{corollary}
\begin{proof}
Omitted.
\end{proof}

This allows us to classify quandles of cyclic type of order $2f$ with $f$ fixed points

\begin{corollary}\label{cor:n=2f}
For any integer $f>1$, there is only one quandle of cyclic type of order $2f$ with $f$ fixed points, up to isomorphism. Moreover, such quandle is not connected.
\end{corollary}
\begin{proof}
Let $f$ be as in the statement. We first note that, should it exist, the indicated quandle is not connected via Proposition \ref{prop:n=2f-notconnected}. We will next prove ($1.$) that  there is such a quandle;  and then ($2.$) that any such quandle is isomorphic to the one in $1.$

\begin{enumerate}
\item Consider the sequence of permutations $$\mu_{i} = (1)(2) \cdots (f) (f+1\quad f+2\quad f+3 \quad \cdots \quad 2f) \qquad \text{ for } i = 1, 2, \dots , f ;$$ $$\mu_{j} = (f+1)(f+2) \cdots (2f) (1\quad 2\quad 3 \quad \cdots \quad f) \qquad \text{ for } j = f+1, f+2, \dots , 2f .$$ Note that $\mu_i$'s and $\mu_j$'s commute among themselves and with one another, since they are either equal or move points from disjoint sets. Then, for any $i, i' \in \{ 1, 2, \dots , f \}$ and $j, j' \in \{ f+1, f+2, \dots , 2f  \}$, we have $$\mu_i(i') = i' \Longrightarrow \mu_{i'} = \mu_i\mu_{i'}\mu_i^{-1} = \mu_{i}$$
$$\mu_i(j) = j+1\,  (\text{with } 2f+1 = f + 1) \Longrightarrow \mu_{j+1} = \mu_i\mu_{j}\mu_i^{-1} = \mu_{j}$$
$$\mu_j(j') = j' \Longrightarrow \mu_{j'} = \mu_j\mu_{j'}\mu_j^{-1} = \mu_{j}$$
$$\mu_j(i) = i+1\,   (\text{with } f+1 = f) \Longrightarrow \mu_{i+1} = \mu_j\mu_{i}\mu_j^{-1} = \mu_{i}.$$ Therefore, the indicated sequence of permutations defines a quandle. Moreover, this is a quandle of cyclic type of order $2f$ with $f$ fixed points.

\item Now consider a quandle of cyclic type of order $2f$ and $f$ fixed points, along with its permutations, $\mu'_i$ for $i = 1, 2, \dots , 2f$. According to $1.$ in Theorem \ref{thr:thethr} and Corollary \ref{cor:thethr} $$\mu'_{2f} = (f+1)(f+2) \cdots (2f) (1\quad 2\quad 3 \quad \cdots \quad f) ,$$ whose set of fixed points is $$F_{2f} = \{  f+1, f+2, \dots , 2f \} = F_{2f-1} = \cdots = F_{f+1}.$$ There are two distinct sets of fixed points (see proof of Proposition \ref{prop:n=2f-notconnected}), so  the other one is $$F_1 = \{  1, 2, \dots , f  \} = F_2 = \cdots = F_f .$$ Then $$\mu'_1 = (1)(2) \cdots (f) (g_{1}\quad g_{2}\quad g_{3} \quad \cdots \quad g_{f}) ,$$ where $(g_{1}\quad g_{2}\quad g_{3} \quad \cdots \quad g_{f})$ is a cyclic permutation of $\{  f+1, f+2, \dots , 2f \}$. For any $i\in \{  1, 2, \dots , f  \}$, we have $$\mu'_{2f}(i) = i+1 \, (\text{with } f+1 = 1) \, \Longrightarrow \mu'_{i+1} = \mu'_{2f}\mu'_i{\mu'_{2f}}^{-1} = \mu'_i .$$ Therefore, $$\mu'_i = \mu'_1 = (1)(2) \cdots (f) (g_{1}\quad g_{2}\quad g_{3} \quad \cdots \quad g_{f}) \qquad \text{ for any } i\in \{  1, 2, \dots , f  \} .$$ Also, for any $i\in \{  1, 2, \dots , f  \}$, we have $$\mu'_{1}(f+i) = f+i+1 \, (\text{with } 2f + 1 = f + 1) \, \Longrightarrow \mu'_{f+i+1} = \mu'_{1}\mu'_{f+i}{\mu'_{1}}^{-1} = \mu'_{f+i} .$$ Therefore, $$\mu'_{f+i} = \mu'_{2f}  = (f+1)(f+2) \cdots (2f) (1\quad 2\quad 3 \quad \cdots \quad f) \qquad \text{ for any } i\in \{  1, 2, \dots , f  \} .$$ Finally, consider the permutation, $\alpha$, of $\{  1, 2, \dots, f, f+1, \dots , 2f \}$ given by $$\alpha = (1)(2)\cdots (f) (g_1\,\,\,  f+1) (g_2\,\,\, f+2) \cdots (g_f\,\,\,  2f) ,$$   where, in case $g_i = f+i$,  $(g_i\,\,\,  f+i)$  is to be read $(g_i)$, a fixed point. Then, for any $i\in \{ 1, 2, \dots , f, f+1, \dots , 2f  \}$, $$\mu'_{\alpha (i)} = \alpha\mu_i\alpha^{-1} .$$ Thus, $\alpha$ is a quandle isomorphism between the quandle here and the quandle in $1.$
\end{enumerate}

\end{proof}

The proof of Corollary \ref{cor:n=2f} establishes Assertion $2.(b)$ in Theorem \ref{thm:main1}.

\subsection{\boldmath Quandles of Cyclic Type of Order $n$ with $f$ Fixed Points in the Range $f+2 \leq n \leq 2f$ are not Connected.}\label{subsect:disconnected-range-f+2-n-2f}

\begin{theorem}
Cyclic quandles of order $n$ with $f$ fixed points such that $f+2\leq n\leq 2f$ are not connected.
\end{theorem}

\begin{proof}
By Proposition \ref{prop:n=2f-notconnected}, we know this is true for $n=2f$. Now, let $Q$ be a cyclic quandle of order $n$ with $f$ fixed points such that $f+2\leq n\leq 2f-1$. We assume, without loss of generality, that $\mu_n=(1\,\,\,\,2 \dots \,\,\,\,n-f)(n-f+1)(n-f+2)\dots (n)$, see Assertion $1.$ in Theorem \ref{thr:thethr}. Given $j,k\in F_n = \{ n-f+1, n-f+2, \dots , n \} $, \[\mu_k(j)=\mu_{\mu_n(k)}(j)=\mu_n\mu_k\mu_n^{-1}(j)=\mu_n(\mu_k(j)),\] that is, $\mu_n$ fixes $\mu_k(j)$. Therefore, we have that $\mu_k(j)\in F_n$ for any $j,k\in F_n$. Thus, $F_n$ is a subquandle of $Q$. Now, if this subquandle has constant profile, then the common pattern is that of $\mu_n\big|_{F_n} =(n-f+1)(n-f+2)\dots (n)$, hence $F_n$ as a quandle is the trivial quandle on $f$ elements. In particular, it is not connected.  If $F_n$ as a  quandle has not constant profile, then, by Proposition \ref{prop:connec}, it is not connected. In either case, this subquandle is not connected and hence there is a finite family of minimal disjoint sets $F_n^i$, $i\in \{ 1, 2, \dots , d \}$, such that $\bigcup_i F_n^i=F_n$ and $\mu_g(F_n^i)=F_n^i$, $\forall i$, $\forall g\in F_n$, which correspond to the (minimal) connected components of $F_n$, as a quandle. We also note that $C_n\cap F_n=\emptyset$, $C_n\cup F_n=Q$ and $\mu_g(C_n)=C_n$, $\forall g\in F_n$. Now, since $|C_n|=n-f<f = |F_1|$, $\mu_{1}$ must fix some point $a_0\in F_n$, i.e., $a_0 \in F_n \cap F_1$.

\begin{lemma}\label{lem:lem}
Let $a\in F_n$. Assume there exist $i_0 \in \{  1, 2, \dots , n-f  \}$ such that:
\begin{enumerate}
  \item $a\in F_{i_0}$. Then $a\in F_i$ for each $i \in \{  1, 2, \dots , n-f  \}$.
  \item $a\notin F_{i_0}$. Then $a\notin F_i$ for each $i \in \{  1, 2, \dots , n-f  \}$.
\end{enumerate}
\end{lemma}
\begin{proof}
\begin{enumerate}
  \item Pick $i \in \{  1, 2, \dots , n-f  \}$ Then $i = \mu_n^{n-f+i-i_0}(i_0)$. Then
\begin{align*}
\mu_i(a) &= \mu_{\mu_n^{n-f+i-i_0}(i_0)}(a) =\mu_n^{n-f+i-i_0}\mu_{i_0}\mu_n^{-(n-f+i-i_0)}(a) =\mu_n^{n-f+i-i_0}\mu_{i_0}(a) =\\
& = \mu_n^{n-f+i-i_0}(a) = a .
\end{align*}
  \item Assume to the contrary and suppose there is  $i_1 \in \{  1, 2, \dots , n-f  \}$ such that $\mu_{i_1}(a)=a$. Then, by $1.$, $\mu_i(a) = a$, for any $i\in \{ 1, 2, \dots , n-f \}$ which conflicts with $\mu_{i_0}(a)\neq a$.
\end{enumerate}
\end{proof}


Thus, the $a_0\in F_1\cap F_n$ above, satisfies, thanks to Lemma \ref{lem:lem}, $a_0\in F_n \cap F_1\cap F_2 \cap \cdots \cap F_{n-f}$.

Set $$A=\{\mu_s(a_0)\,\,|\, \, s\in \{ 1, 2, \dots , n \}\, \}=\{\mu_j(a_0)\,\,|\, \, j\in \{ n-f+1, n-f+2, \dots , n \}\, \}\subset F_n^1 ,$$ since $a_0 \in F_n \cap F_1 \cap F_2 \cap \cdots \cap F_{n-f}$. In particular, $1\leq |A| \leq f$.

For each $i\in \{ 1, 2, \dots , n-f  \}$, here is the behaviour of $A$ under $\mu_i$. Let $j\in \{  n-f+1, n-f+2, \dots , n   \}$ $$\mu_i(\mu_j(a_0)) = \mu_i(\mu_j\mu_i^{-1}(\mu_i((a_0)))) = \mu_i(\mu_j\mu_i^{-1}((a_0))) = \mu_{\mu_i(j)}(a_0) \in A .$$ Furthermore, for $j, j'\in \{  n-f+1, n-f+2, \dots , n   \}$ such that $\mu_j (a_0) \neq \mu_{j'}(a_0)$, then $\mu_i(\mu_j (a_0)) \neq \mu_i(\mu_{j'}(a_0))$, for any $i\in \{  1, 2, \dots , n-f  \}$. Then, for any $i\in \{  1, 2, \dots , n-f  \}$, $\mu_i$ restricted to $A$ is a bijection.

\begin{enumerate}
  \item If $A=\{ a_0 \}$, then $\mu_s(a_0) = a_0$, for any $s\in \{  1, 2, \dots , n \}$, so $A$ is a connected component of $Q$. Since $Q$ has more than one element then $Q$ is not connected.
  \item Assume $|A| > 1$.
  \begin{enumerate}
    \item Assume further that, for each $i\in \{ 1, 2, \dots , n-f \}$, $\mu_i$ moves at least on element from $A$, say $\mu_{j_0}(a_0)$, for some $j_0 \in \{ n-f+1, n-f+2, \dots , n  \}$ - recall Lemma \ref{lem:lem}. Then, $$\mu_i(\mu_{j_0}(a_0)) \in A \setminus \{ \mu_{j_0}(a_0) \} .$$ Since $\mu_i$ has a cycle of length $n-f$, $C_i$, then $C_i\subset A$, for any $i\in \{ 1, 2, \dots , n-f  \}$. In particular, in this case, if $b\notin A$, then $\mu_i(b) = b$, for any $i\in \{ 1, 2, \dots , n-f  \}$.

        Since $A\subset F_n^1$, then $F_n^1$ is a connected component of $Q$. Since $C_n \cap F_n^1 = \emptyset$ and $C_n \subset Q$, then $Q$ is not connected.
    \item Assume now that for each $i\in \{ 1, 2, \dots , n-f  \}$, $\mu_i$ fixes any element of $A$ - again, recall Lemma \ref{lem:lem}. That is, for each $j\in \{ n-f+1, n-f+2, \dots , n  \}$, $$ \mu_i(\mu_j(a_0)) = \mu_j(a_0) .$$
  \begin{enumerate}
    \item If $A = F_n^1$, then we are done, arguing that $F_n^1$ is a connected component inside $Q$ which has fewer elements than $Q$.
    \item If $A \subsetneq F_n^1$, then since $F_n^1$ is a minimal component of $F_n$, there exist $j_0, j_1\in \{  n-f+1, n-f+2, \dots n   \}$ such that $$\mu_{j_1}(\mu_{j_0}(a_0))\in F_n^1\setminus A .$$  Let $B_1 = \{ \mu_{j_1}(\mu_{j_0}(a_0))\in F_n^1\setminus A  \, \, |\, \,  j_0, j_1\in \{  n-f+1, n-f+2, \dots n   \}  \}$. Note that if for some $i\in \{ 1, 2, \dots , n-f  \}$, and for some $j_0, j_1\in \{  n-f+1, n-f+2, \dots n   \}$ such that $\mu_{j_1}(\mu_{j_0}(a_0))\in F_n^1\setminus A$, we had $\mu_i(\mu_{j_1}(\mu_{j_0}(a_0)))\in A$, then $\mu_{j_1}(\mu_{j_0}(a_0))\in A$, since $\mu_i$ restricted to $A$ is a bijection. Therefore, $\mu_i$ restricted to $B_1$ is a bijection.

        Set $$A_1:=A \cup B_1 .$$ Go back to $2.$ (``Assume $|A|>1$ ... '') with $A_1$ taking up the role of $A$. Iterate the procedure. Since $Q$ is a finite quandle, this procedure has to finish after a finite number of steps, say $k$, with $A(=A_k) = F_n^1$.
  \end{enumerate}
  \end{enumerate}
\end{enumerate}

The proof is complete.
\end{proof}

\section{\boldmath Classification of Quandles of Cyclic Type of order $n$ with $f$ Fixed Points in the Range $n>2f$.}\label{sect:n>2f}

In this Section, we classify quandles of cyclic type of order $n$ with $f$ fixed points such that $n>2f$. Specifically, we prove that there is only one such quandle such that $n>2f$, up to isomorphism. In this range, this quandle occurs only for $n=6$ and $f=2$. This quandle is  $Q_6^2$, up to isomorphism. The proof of this fact establishes Assertion $1.(b)$ in Theorem \ref{thm:main1}. This is the main goal of the current Section.\par
In Subsection \ref{subsect:auxresults}, we prove a number of propositions and lemmas that we use in subsequent subsections. In Subsection \ref{subsect:n=3f}, we show that there are no quandles of cyclic type of order $n$ with $f$ fixed points such that $n=3f$ for $f>2$ and we prove that the only quandle of cyclic type of order $6$ with $2$ fixed points, up to isomorphism, is $Q_6^2$. In Subsection \ref{subsect:n>3f}, we show that there are no quandles of cyclic type of order $n$ with $f$ fixed points such that $n=cf$ for $c>3$. Finally, in Subsection \ref{subsect:4}, we collect the results from the preceding subsections to prove Assertion $1.(b)$ in Theorem \ref{thm:main1}. We also show that $Q_6^2$ is not simple.\par
In this Section, the results apply only to quandles of cyclic type of order $n$ with $f$ fixed points such that $n>2f$.

\subsection{Auxiliary Results}\label{subsect:auxresults}

In this Subsection, we state and prove a number of results about the structure of quandles of cyclic type of order $n$ with $f$ fixed points such that $n>2f$. These results are used in the following Subsections.\par

\begin{prop}\label{prop:assind}
The associate indices to $n-f$ are $i\Big(\frac{n-f}{f}\Big)$, $i\in\{1, \dots, f-1\}$.
\end{prop}

\begin{proof}
We first prove that, if $a, b\in F_{n-f}$, then $b-a\in F_{n-f} \mod n-f$. Since $a, b\in F_{n-f}$, $\mu_a$ fixes $b$ (Proposition \ref{prop:assocind}) and hence, reading $b-a \mod n-f$ and using Assertion $3.$ in Theorem \ref{thr:thethr}, \[b=\mu_a(b)=\mu_n^a\mu_{n-f}\mu_n^{-a}(b)=\mu_n^a\mu_{n-f}(b-a)=\mu_{n-f}(b-a)+a\Leftrightarrow b-a=\mu_{n-f}(b-a),\] that is, $\mu_{n-f}$ fixes $b-a$, implying that $b-a\in F_{n-f}$ (where $b-a$ is read modulo $n-f$). Now, let the indices in $F_{n-f}=\{g_{n-f}^1, \dots, g_{n-f}^f\}$ be labelled in such a way that $g_{n-f}^i<g_{n-f}^{i+1}$, $\forall i\in\{1, \dots, f-1\}$. In particular, we have $g_{n-f}^f=n-f$. Suppose these indices are not equally spaced modulo $n-f$. Therefore, there is an index $1\leq j_0\leq f$ such that $g_{n-f}^{j_0}-g_{n-f}^{{j_0}-1}\geq g_{n-f}^i-g_{n-f}^{i-1}$, $\forall i\in\{1, \dots, f\}$, and there is another index $1\leq k_0\leq f$ such that $g_{n-f}^{k_0}-g_{n-f}^{{k_0}-1}\leq g_{n-f}^i-g_{n-f}^{i-1}$, $\forall i\in\{1, \dots, f\}$, where we take $g_{n-f}^0:=0$. Moreover, $g_{n-f}^{j_0}-g_{n-f}^{{j_0}-1}>g_{n-f}^{k_0}-g_{n-f}^{{k_0}-1}$. Now, by the result we have just proved, $g_{n-f}^{k_0}-g_{n-f}^{{k_0}-1}$ belongs to $F_{n-f}$, as well as $g_{n-f}^{j_0}-(g_{n-f}^{k_0}-g_{n-f}^{{k_0}-1})$. However, we have $g_{n-f}^{{j_0}-1}<g_{n-f}^{j_0}-(g_{n-f}^{k_0}-g_{n-f}^{{k_0}-1})<g_{n-f}^{j_0}$, which is a contradiction, since $g_{n-f}^{{j_0}-1}$ and $g_{n-f}^{j_0}$ are consecutive indices in $F_{n-f}$. Hence, the indices in $F_{n-f}$ are equally spaced modulo $n-f$, and the associate indices to $n-f$ are $i\Big(\frac{n-f}{f}\Big)$, $i\in\{1, \dots, f-1\}$.
\end{proof}

\begin{corollary}\label{cor:assind}
For each index $i\in\{1,\dots,n-f\}$, $F_i=\Big\{i+j\Big(\frac{n-f}{f}\Big):1\leq j\leq f \Big\}=i+F_{n-f}$, where each index $i+j\Big(\frac{n-f}{f}\Big)$ is read modulo $n-f$.
\end{corollary}

\begin{proof}
Given $i\in\{1,\dots,n-f\}$ and $j\in\{1,\dots,f\}$, we prove $i+j\Big(\frac{n-f}{f}\Big)$ is a fixed point of $\mu_i$. In fact, \[\mu_i\Big(i+j\Big(\tfrac{n-f}{f}\Big)\Big)=\mu_n^i\mu_{n-f}\mu_n^{-i}\Big(i+j\Big(\tfrac{n-f}{f}\Big)\Big)=\mu_n^i\mu_{n-f}\Big(j\Big(\tfrac{n-f}{f}\Big)\Big)=\mu_n^i\Big(j\Big(\tfrac{n-f}{f}\Big)\Big)=i+j\Big(\tfrac{n-f}{f}\Big).\] Clearly, for $1\leq j<j'\leq f$, we have $i+j\Big(\tfrac{n-f}{f}\Big)\neq i+j'\Big(\tfrac{n-f}{f}\Big) \mod n-f$. Hence, we conclude that $F_i=\Big\{i+j\Big(\frac{n-f}{f}\Big):1\leq j\leq f \Big\}$, where each $i+j\Big(\frac{n-f}{f}\Big)$ is read modulo $n-f$.
\end{proof}

Corollary \ref{cor:assind} along with Proposition \ref{prop:assocind} and Corollary \ref{cor:order} tell us exactly what are the associate indices in a quandle of cyclic type with $f$ fixed points of order $n>2f$.

\begin{lemma}\label{lem:dif1}
Given distinct indices $a,b\in F_n$, $\mu_{n-f}(b)-\mu_{n-f}(a)\in F_{n-f}$, where this index is read modulo $n-f$.
\end{lemma}

\begin{proof}
Let $a,b\in F_n$. By assertion \textit{2.} in Theorem \ref{thr:thethr}, we have $\mu_b=\mu_a^{l_{b,a}}$. Hence, using assertion \textit{4.} in Theorem \ref{thr:thethr}, \[\mu_n^{\mu_{n-f}(a)}\mu_{n-f}^{l_{b,a}}\mu_n^{-\mu_{n-f}(a)}=\Big(\mu_n^{\mu_{n-f}(a)}\mu_{n-f}\mu_n^{-\mu_{n-f}(a)}\Big)^{l_{b,a}}=\Big(\mu_{n-f}\mu_a\mu_{n-f}^{-1}\Big)^{l_{b,a}}=\]\[=\mu_{n-f}\mu_a^{l_{b,a}}\mu_{n-f}^{-1}=\mu_{n-f}\mu_b\mu_{n-f}^{-1}=\mu_n^{\mu_{n-f}(b)}\mu_{n-f}\mu_n^{-\mu_{n-f}(b)},\] which implies, by assertion \textit{3.} in Theorem \ref{thr:thethr}, that \[\mu_{n-f}^{l_{b,a}}=\mu_n^{\mu_{n-f}(b)-\mu_{n-f}(a)}\mu_{n-f}\mu_n^{-(\mu_{n-f}(b)-\mu_{n-f}(a))}=\mu_{\mu_{n-f}(b)-\mu_{n-f}(a)},\] where the index $\mu_{n-f}(b)-\mu_{n-f}(a)$ is read modulo $n-f$. Since $1\leq l_{b,a}<n-f$, then $\mu_{n-f}^{l_{b,a}}$ and $\mu_{n-f}$ have the same set of fixed points. Thus, the equalities above imply that, modulo $n-f$, $\mu_{n-f}(b)-\mu_{n-f}(a)\in F_{n-f}$.
\end{proof}

\begin{corollary}\label{cor:fplets}
$\mu_{n-f}(F_n)=F_k$, for some $k\notin F_n\cup F_{n-f}$.
\end{corollary}

\begin{proof}
Let $a\in F_n$ and let $\mu_{n-f}(a)=k$, where $k\notin F_n\cup F_{n-f}$ by Lemma \ref{lemma:ups}. For each $b\in F_n\setminus\{a\}$, we have that $\mu_{n-f}(b)-\mu_{n-f}(a)\in F_{n-f}$ by Lemma \ref{lem:dif1}. Therefore, by Corollary \ref{cor:assind}, $\mu_{n-f}(F_n)\subset F_k$. Since $|F_n|=|F_k|$ and $\mu_{n-f}$ is a bijection, $\mu_{n-f}(F_n)=F_k$.
\end{proof}

\begin{corollary}\label{cor:assper}
Let $\mu_{n-f}(F_n)=F_k$, $k\notin F_n\,\cup\,F_{n-f}$, see Corollary \ref{cor:fplets}. Then $\mu_{i\big(\frac{n-f}{f}\big)}(F_n)=F_k$, $\forall i\in\{1, \dots, f\}$.
\end{corollary}

\begin{proof}
Given $a\in F_n$ and $i\in\{1,\dots,f\}$, we prove that $\mu_{i\big(\frac{n-f}{f}\big)}(a)\in F_k$. We have \[\mu_{i\big(\frac{n-f}{f}\big)}(a)=\mu_n^{i\big(\frac{n-f}{f}\big)}\mu_{n-f}\mu_n^{-i\big(\frac{n-f}{f}\big)}(a)=\mu_n^{i\big(\frac{n-f}{f}\big)}\mu_{n-f}(a)=\mu_{n-f}(a)+i\Big(\tfrac{n-f}{f}\Big) \in F_{\mu_{n-f}(a)} = F_k.\] 
by Corollary \ref{cor:assind}. Hence, $\mu_{i\big(\frac{n-f}{f}\big)}(F_n)\subset F_k$. As $|F_n|=|F_k|$ and $\mu_{i\big(\frac{n-f}{f}\big)}$ is a bijection, $\mu_{i\big(\frac{n-f}{f}\big)}(F_n)=F_k$ and the result follows.
\end{proof}

\begin{corollary}\label{cor:difper}
$\mu_{i\big(\frac{n-f}{f}\big)}\neq\mu_{j\big(\frac{n-f}{f}\big)}$, where $1\leq i\neq j\leq f$.
\end{corollary}

\begin{proof}
Given $a\in F_n$, $\mu_{i\big(\frac{n-f}{f}\big)}(a)=\mu_{n-f}(a)+i\Big(\frac{n-f}{f}\Big)$, $\forall i\in\{1,\dots,f\}$. Then, given $i,j\in\{1,\dots,f\}$, with $i\neq j$, $\mu_{i\big(\frac{n-f}{f}\big)}\neq\mu_{j\big(\frac{n-f}{f}\big)}$.
\end{proof}

With the previous results, we are now able to prove a very important proposition. In fact, this proposition is used in the following section to prove that there are no quandles of cyclic type of order $n$ with $f$ fixed points such that $n=3f$ for $f>2$. \par
Before presenting the proposition, we just recall some of the terminology we used in assertion \textit{2.} in Theorem \ref{thr:thethr}.

\begin{definition}\label{def:expassociates}
Let $h$ and $h'$ be associate indices. We let $l_{h, h'}$ denote the positive integer such that $\mu_h = \mu_{h'}^{l_{h, h'}}$, where $GCD(n-f, l_{h,h'})=1$, $1\leq l_{h,h'}<n-f$, in accordance with assertion \textit{2.} in Theorem \ref{thr:thethr}.
Moreover, since the associate indices to $n-f$ are $i\Big(\frac{n-f}{f}\Big)$, $i\in \{1,\dots,f-1\}$, we let $l_{i, f}^{\ast}:=l_{\frac{i(n-f)}{f}, n-f}$ to simplify the notation.
\end{definition}

\begin{prop}\label{prop:exp}
$\{l_{i,f}^*:1\leq i\leq f\}=\Big\{1+j\Big(\frac{n-f}{f}\Big):0\leq j<f\Big\}$
\end{prop}

\begin{proof}
By Corollary \ref{cor:fplets}, there exists an index $k\notin F_n \cup F_{n-f}$ such that $\mu_{n-f}(F_n) = F_k$. Moreover, we write $F_k = \{g_k^1, \dots,g_k^f \}$, where the indices $g_k^i$ are labelled in such a way that there exist positive integers $l_i>1$, $\forall i\in\{1,\dots,f\}$, such that $\mu_{n-f}^{l_i}(g_k^{i-1})=g_k^i$ and $\mu_{n-f}^m(g_k^{i-1})\notin F_k$ for $1<m<l_i$, where we take $g_k^0:=g_k^f$. Basically, each $l_i$ is the smallest positive integer such that $\mu_{n-f}^{l_i}(g_k^{i-1})\in F_k$, and we have $\mu_{n-f}^{l_i}(g_k^{i-1})=g_k^i$. We note that the indices in $F_k$ belong to the non-singular cycle of $\mu_{n-f}$ and we prove the following lemma.

\begin{lemma}\label{lemma:exp}
The indices in $F_k$ are equally spaced in the non-singular cycle of $\mu_{n-f}$. In particular, the $l_i$'s referred to above satisfy $l_i=(n-f)/f$, $\forall i\in\{1,\dots,f\}$.
\end{lemma}

\begin{proof}
Suppose this is not true. Then, there is an index $j_0\in\{1,\dots,f\}$ such that $l_{j_0}<l_{j'_0}$, for a certain $j'_0\in\{1,\dots,f\}$. Assume, without loss of generality, that the indices in $F_n=\{g_n^1,\dots,g_n^f\}$ are labelled in such a way that $\mu_{n-f}(g_n^i)=g_k^i$, $\forall i\in\{1,\dots,f\}$. Taking $g_n^0:=g_n^f$,

\begin{equation}\label{eq:1}
\mu_{n-f}^{l_{j_0}}(g_k^{{j_0}-1})=g_k^{j_0}\Leftrightarrow\mu_{n-f}^{l_{j_0}}(\mu_{n-f}(g_n^{{j_0}-1}))=g_k^{j_0}\Leftrightarrow\mu_{n-f}^{l_{j_0}+1}(g_n^{{j_0}-1})=g_k^{j_0}.\end{equation}

By Corollary \ref{cor:assper}, $\mu_{n-f}$ and its associate permutations are bijections from $F_n$ to $F_k$, where $|F_k|\,\,=f$. These $f$ permutations are of the form $\mu_{i\big(\frac{n-f}{f}\big)}=\mu_{n-f}^{l_{i,f}^*}$, with $i\in\{1,\dots,f\}$ and $l_{f,f}^*=1$, and they are all different from each other by Corollary \ref{cor:difper}. Hence, any two of them have different images at $g_n^{{j_0}-1}$, otherwise at least two of these permutations would be equal to each other, conflicting with Corollary \ref{cor:difper}. (in particular, $\{   l_{i, f}^{\ast} : 1 \leq i \leq f  \}$ has exactly $f$ elements). Then, there has to be an integer $j_1\in\{1,\dots,f\}$ such that $\mu_{j_1\big(\frac{n-f}{f}\big)}(g_n^{{j_0}-1})=\mu_{n-f}^{l_{j_1,f}^*}(g_n^{{j_0}-1})=g_k^{j_0}$. Hence, comparing with (\ref{eq:1}), $l_{j_1,f}^*=l_{j_0}+1$ and $\mu_{j_1\big(\frac{n-f}{f}\big)}=\mu_{n-f}^{l_{j_0}+1}$, and thus \[\mu_{j_1\big(\frac{n-f}{f}\big)}(g_n^{{j'_0}-1})=\mu_{n-f}^{l_{j_0}+1}(g_n^{{j'_0}-1})=\mu_{n-f}^{l_{j_0}}(\mu_{n-f}(g_n^{{j'_0}-1}))=\mu_{n-f}^{l_{j_0}}(g_k^{{j'_0}-1}).\] However, $\mu_{n-f}^{l_{j_0}}(g_k^{{j'_0}-1})\notin F_k$, as $1<l_{j_0}<l_{{j'_0}}$. This is a contradiction, since $\mu_{j_0\big(\frac{n-f}{f}\big)}(F_n)=F_k$. Thus, the indices in $F_k$ are equally spaced in the cycle of length $n-f$ of $\mu_{n-f}$. \end{proof}

We now resume the proof of Proposition \ref{prop:exp}. Given $g_k^{j}\in F_k$, there is an index $g_n^j\in F_n$ such that $\mu_{n-f}(g_n^j)=g_k^{j}$. Then, $\forall i\in\{1,\dots,f\}$, \[F_k\ni\mu_{i\big(\frac{n-f}{f}\big)}(g_n^j)=\mu_{n-f}^{l_{i,f}^*}(g_n^j)=\mu_{n-f}^{l_{i,f}^*-1}(\mu_{n-f}(g_n^j))=\mu_{n-f}^{l_{i,f}^*-1}(g_k^j),\] that is, $\forall i\in\{1,\dots,f\}$, we have $\mu_{n-f}^{l_{i,f}^*-1}(F_k)\in F_k$. Since the $f$ indices in $F_k$ are equally spaced in the cycle of length $n-f$ of $\mu_{n-f}$, $f\big(l_{i,f}^*-1\big)$ has to be a multiple of $n-f$, that is, $f\big(l_{i,f}^*-1\big)=k_i(n-f)$, for some natural number $0\leq k_i<f$, $\forall i\in\{1,\dots,f\}$. This is equivalent to saying that \[f(l_{i,f}^*-1)=k_i(n-f)\Leftrightarrow l_{i,f}^*-1=k_i\Big(\tfrac{n-f}{f}\Big)\Leftrightarrow l_{i,f}^*=1+k_i\Big(\tfrac{n-f}{f}\Big),\] for some natural number $0\leq k_i<f$, $\forall i\in\{1,\dots,f\}$. Hence, the set $\{l_{i,f}^*:1\leq i\leq f\}$ is contained in $\Big\{1+j\Big(\frac{n-f}{f}\Big):0\leq j<f\Big\}$. Since these two sets must have the same cardinality, the result follows.
\end{proof}

\begin{remark}
We remark that the integers in $\{l_{i,f}^*:1\leq i\leq f\}$ still have to verify the conditions presented in assertion \textit{2.} in Theorem \ref{thr:thethr}, that is, $GCD(n-f, l_{i,f}^*)=1$, $1\leq l_{i,f}^*<n-f$, $\forall i\in\{1,\dots,f\}$. In fact, there are certain pairs of integers $(n,f)$ for which these conditions are not satisfied if $l_{i, f}^{\ast}$ has the form $1+\tfrac{i(n-f)}{f}$ with $0\leq i<f$, where $l_{0,f}^*:=l_{f,f}^*$. Thus the corresponding quandles of cyclic type of order $n$ with $f$ fixed points cannot exist. For example, there cannot be quandles of cyclic type of order $28$ with $7$ fixed points, but there can be quandles of cyclic type of order $6$ with $2$ fixed points. Therefore, from now on, we are only working with pairs of integers $(n,f)$ for which these conditions are satisfied.
\end{remark}

\subsection{\boldmath Quandles of Cyclic Type of Order $n$ with $f$ Fixed Points such that $n=3f$}\label{subsect:n=3f}

In this Subsection, we show there are no quandles of cyclic type of order $n$ with $f$ fixed points such that $n=3f$ for $f>2$. We also prove that the only quandle of cyclic type of order $6$ with $2$ fixed points, up to isomorphism, is $Q_6^2$.\par
Firstly, we prove a result which is a direct consequence of Proposition \ref{prop:exp}.

\begin{corollary}\label{cor:invperm}
Given a quandle of cyclic type $Q$ of order $n$ with $f$ fixed points such that $n=3f$, $\mu_{n-f}^{-1}=\mu_{2i}$, for a certain integer $i\in \{1,\dots,f\}$.
\end{corollary}

\begin{proof}
By Proposition \ref{prop:exp}, $\{l_{i,f}^*:1\leq i\leq f\}=\Big\{1+j\Big(\frac{n-f}{f}\Big):0\leq j<f\Big\}$. Taking $j=f-1$, we conclude there exists an integer $i'\in \{1,\dots,f\}$ such that $l_{i',f}^*=1+(f-1)\Big(\frac{n-f}{f}\Big)$. As $n=3f$, we get \[l_{i',f}^*=l_{\frac{i'(n-f)}{f}, n-f}=l_{2i',2f}=1+2(f-1)=2f-1=(n-f)-1\equiv -1\mod n-f,\] that is, there exists an integer $i'\in \{1,\dots,f\}$ such that $\mu_{n-f}^{-1}=\mu_{2f}^{-1}=\mu_{2i'}$.
\end{proof}

Corollary \ref{cor:invperm} tells us that, for any quandle of cyclic type $Q$ of order $n$ with $f$ fixed points such that $n=3f$, $\mu_{n-f}^{-1}$ is a permutation of $Q$ and it is of the form $\mu_{n-f}^{-1}=\mu_{2f}^{-1}=\mu_{2i}$, for a certain integer $i\in \{1,\dots,f\}$. We now prove a lemma, similar to Lemma \ref{lem:dif1}, which also has some very useful consequences.

\begin{lemma}\label{lem:dif2}
Given distinct indices $a,b\in F_n$, $\mu_{n-f}^{-1}(b)-\mu_{n-f}^{-1}(a)\in F_{n-f}$, where this index is read modulo $n-f$.
\end{lemma}

\begin{proof}
Let $a,b\in F_n$. By assertion \textit{2.} in Theorem \ref{thr:thethr}, we have $\mu_b=\mu_a^{l_{b,a}}$. Hence, using assertion \textit{5.} in Theorem \ref{thr:thethr}, \[\mu_n^{\mu_{n-f}^{-1}(a)}\mu_{n-f}^{l_{b,a}}\mu_n^{-\mu_{n-f}^{-1}(a)}=\Big(\mu_n^{\mu_{n-f}^{-1}(a)}\mu_{n-f}\mu_n^{-\mu_{n-f}^{-1}(a)}\Big)^{l_{b,a}}=\Big(\mu_{n-f}^{-1}\mu_a\mu_{n-f}\Big)^{l_{b,a}}=\]\[=\mu_{n-f}^{-1}\mu_a^{l_{b,a}}\mu_{n-f}=\mu_{n-f}^{-1}\mu_b\mu_{n-f}=\mu_n^{\mu_{n-f}^{-1}(b)}\mu_{n-f}\mu_n^{-\mu_{n-f}^{-1}(b)},\] which implies, by assertion \textit{3.} in Theorem \ref{thr:thethr}, that \[\mu_{n-f}^{l_{b,a}}=\mu_n^{\mu_{n-f}^{-1}(b)-\mu_{n-f}^{-1}(a)}\mu_{n-f}\mu_n^{-(\mu_{n-f}^{-1}(b)-\mu_{n-f}^{-1}(a))}=\mu_{\mu_{n-f}^{-1}(b)-\mu_{n-f}^{-1}(a)},\] where the index $\mu_{n-f}^{-1}(b)-\mu_{n-f}^{-1}(a)$ is read modulo $n-f$. Since $1\leq l_{b,a}<n-f$, then $\mu_{n-f}^{l_{b,a}}$ and $\mu_{n-f}$ have the same set of fixed points. Thus, the equalities above imply that, modulo $n-f$, $\mu_{n-f}^{-1}(b)-\mu_{n-f}^{-1}(a)\in F_{n-f}$.
\end{proof}

\begin{corollary}\label{cor:dif3}
Given $a,b\in F_n$, $\mu_{n-f}(b)-\mu_{n-f}(a)=\mu_{n-f}^{-1}(b)-\mu_{n-f}^{-1}(a)$, where these two indices are read modulo $n-f$.
\end{corollary}

\begin{proof}
Any two permutations from $\mu_1$ to $\mu_{n-f}$ are distinct. Indeed, given an index $a\in F_n$ and for each index $k\in\{1,\dots,n-f\}$, we have, by assertion \textit{3.} in Theorem \ref{thr:thethr}, \[\mu_k(a)=\mu_n^k\mu_{n-f}\mu_n^{-k}(a)=\mu_n^k\mu_{n-f}(a)=\mu_{n-f}(a)+k.\] Since we also have, by Lemmas \ref{lem:dif1} and \ref{lem:dif2}, that \[\mu_{\mu_{n-f}(b)-\mu_{n-f}(a)}=\mu_{n-f}^{l_{a,b}}=\mu_{\mu_{n-f}^{-1}(b)-\mu_{n-f}^{-1}(a)},\] where the indices $\mu_{n-f}(b)-\mu_{n-f}(a)$ and $\mu_{n-f}^{-1}(b)-\mu_{n-f}^{-1}(a)$ are read modulo $n-f$, we get \[\mu_{n-f}(b)-\mu_{n-f}(a)=\mu_{n-f}^{-1}(b)-\mu_{n-f}^{-1}(a),\] where these two indices are read modulo $n-f$.
\end{proof}

We now have all the results we need to prove there are no quandles of cyclic type of order $n$ with $f$ fixed points such that $n=3f$ for $f>2$. This is the result we state in the following proposition.

\begin{prop}\label{prop:cyclic1}
There are no quandles of cyclic type of order $n$ with $f$ fixed points such that $n=3f$ for $f>2$.
\end{prop}

\begin{proof}
Suppose $Q$ is a quandle of cyclic type of order $n$ with $f$ fixed points such that $n=3f$. First of all, we note that $Q=F_{1}\cup F_{n-f}\cup F_{n}$. We know $\mu_{n-f}(F_{n-f})=F_{n-f}$ and we have $\mu_{n-f}(F_n)=F_1$ by Corollary \ref{cor:invperm}. Since $\mu_{n-f}$ is a bijection, we have $\mu_{n-f}(F_1)=F_n$. Now, let $a\in F_n$ and suppose $\mu_{n-f}(a)=j$ and $\mu_{n-f}^{-1}(a)=k$, where $j,k\in F_1$. We note that $j\neq k$, otherwise $\mu_{n-f}$ would have a cycle of length $2$. This is not possible as the length of this cycle is $n-f=3f-f=2f\geq 4$. For each $i\in\{1,\dots,n-f\}$, we have by assertion \textit{3.} in Theorem \ref{thr:thethr}, reading the indices modulo $n-f$,
\begin{equation}\label{eq:2}
\mu_i(a)=\mu_n^i\mu_{n-f}\mu_n^{-i}(a)=\mu_n^i\mu_{n-f}(a)=\mu_n^i(j)=j+i.
\end{equation}
and, again by assertion \textit{3.} in Theorem \ref{thr:thethr}, reading the indices modulo $n-f$,
\begin{equation}\label{eq:3}
\mu_i(k+i)=\mu_n^i\mu_{n-f}\mu_n^{-i}(k+i)=\mu_n^i\mu_{n-f}(k)=\mu_n^i(a)=a,
\end{equation}
By Corollary \ref{cor:invperm}, we know that there is an index $i'\in\{1,\dots,f\}$ such that $\mu_{2i'}=\mu_{n-f}^{-1}$. Therefore, we have that $\mu_{2i'}(a)=k$ and $\mu_{2i'}(j)=a$. Since $\mu_{2i'}(a)=j+2i'$ by (\ref{eq:2}), we conclude that $k=j+2i'$. Hence, $\mu_{2i'}(j)=\mu_{2i'}(k-2i')$ and we also conclude $\mu_{2i'}(k-2i')=a$. But $\mu_{2i'}(k+2i')=a$ by (\ref{eq:3}), which means that we have either $2i'=n-f$ or $2i'=(n-f)/2$. Since $2i'$ cannot be equal to $n-f$, as this forces $k$ to be equal to $j$, we then have $2i'=(n-f)/2$. Now, suppose $\mu_{n-f}(j)=b\in F_n$, with $b\neq a$. Noting that we have $\mu_{n-f}(a)-\mu_{n-f}(b)=\mu_{n-f}^{-1}(a)-\mu_{n-f}^{-1}(b)$ by Corollary \ref{cor:dif3}, we get \[\mu_{n-f}^{-1}(b)-\mu_{n-f}(b)=\mu_{n-f}^{-1}(a)-\mu_{n-f}(a)=k-j=\tfrac{n-f}{2}.\] Since $\mu_{n-f}^{-1}(b)=j$, this forces $\mu_{n-f}(b)$ to be equal to $k$. Therefore, $\mu_{n-f}$ has a cycle $\tau$ of length $4$, where $\tau=(a\,\,j\,\,b\,\,k)$. Hence, for any quandle of cyclic type of order $n$ with $f$ fixed points such that $n=3f$, $f=2$ and $n=6$ and, thus, there are no quandles of cyclic type of order $n$ with $f$ fixed points such that $n=3f$ for $f>2$.
\end{proof}

\begin{corollary}\label{cor:q62}
There is only one quandle of cyclic type of order $6$ with $2$ fixed points, up to isomorphism.
\end{corollary}

\begin{proof}
$\mu_6=(1\,\,2\,\,3\,\,4)(5)(6)$, up to isomorphism, by assertion \textit{1.} in Theorem \ref{thr:thethr} and then $F_5=\{5,6\}=F_6$. $F_2=\{2,4\}=F_4$ by Proposition \ref{prop:assind}, and so $F_1=\{1,3\}=F_3$. Now, by Corollary \ref{cor:fplets}, we can either have $\mu_4=(2)(4)(1\,\,6\,\,3\,\,5)$ or $\mu_4=(2)(4)(1\,\,5\,\,3\,\,6)$. However, straightforward calculations show that the latter does not satisfy assertion \textit{4.} in Theorem \ref{thr:thethr}, and hence $\mu_4=(2)(4)(1\,\,6\,\,3\,\,5)$. By Theorem \ref{thr:thethr}, we can write $\mu_1$, $\mu_2$, $\mu_3$ and $\mu_5$ as functions of $\mu_6$ and $\mu_4$, thus $\mu_4$ determines one single quandle, which is precisely $Q_6^2$. Its multiplication table is displayed in Table \ref{table:1}.
\end{proof}

\subsection{\boldmath Quandles of Cyclic Type of Order $n$ with $f$ Fixed Points such that $n>3f$}\label{subsect:n>3f}
In this Subsection, we use the results from the previous Subsections to show that there are no quandles of cyclic type of order $n$ with $f$ fixed points such that $n=cf$ for $c>3$. This is the result we state in the following proposition, where in its proof the free indices are read modulo $n-f$.

\begin{prop}\label{prop:cyclic2}
There are no quandles of cyclic type of order $n$ with $f$ fixed points such that $n=cf$, for $c>3$.
\end{prop}

\begin{proof}
Let $\mu_{n-f}(F_n)=F_k$, in accordance with Corollary \ref{cor:fplets}. We know the indices in $F_k$ belong to the cycle of length $n-f$ of $\mu_{n-f}$, they are all equally spaced in this cycle by Lemma \ref{lemma:exp} and their inverse images are in $F_n$. Therefore, the indices in $F_n$ are also equally spaced in the cycle of length $n-f$ of $\mu_{n-f}$. Now, we have that for each $m\in\{1,\dots, n-f\}\setminus\{\mu_{n-f}^{-1}(n-f+1), \dots,\mu_{n-f}^{-1}(n)\}$, there is an integer $1\leq k_m < n-f$ such that $\mu_n^{-\mu_{n-f}(m)}\mu_{n-f}\mu_n^m=\sigma\tau^{k_m}$, by assertion \textit{6.} in Theorem \ref{thr:thethr}. We now prove that each $m$ gives rise to a different $k_m$. In fact, given $a\in F_n$, \[\tau^{k_m}(a)=\sigma\tau^{k_m}(a)=\mu_{n-f}^{k_m}(a)=\mu_n^{-\mu_{n-f}(m)}\mu_{n-f}\mu_n^m(a)=\mu_{n-f}(a)-\mu_{n-f}(m),\] which has a different value for each $m\in\{1,\dots, n-f\}\setminus\{\mu_{n-f}^{-1}(n-f+1), \dots,\mu_{n-f}^{-1}(n)\})$. Therefore, assertion \textit{6.} in Theorem \ref{thr:thethr} provides us with a total of $n-2f$ different equalities and, in particular, $n-2f$ different integers $k_m$. Now, given $a,b\in F_n$, we can combine the two equalities \[\mu_{n-f}(a)-\mu_{n-f}(m)=\mu_{n-f}^{k_m}(a),\] \[\mu_{n-f}(b)-\mu_{n-f}(m)=\mu_{n-f}^{k_m}(b),\] in order to get \[F_{n-f}\ni\mu_{n-f}(a)-\mu_{n-f}(b)=\mu_{n-f}^{k_m}(a)-\mu_{n-f}^{k_m}(b).\] Indeed, we know exactly what are the $n-2f$ different integers $k_m$ which satisfy this equality. Suppose that $k_m\in\Big\{i\Big(\frac{n-f}{f}\Big):1\leq i\leq f\Big\}=F_{n-f}$. Therefore, $\mu_{n-f}^{k_m}(a),\mu_{n-f}^{k_m}(b)\in F_n$ by Lemma \ref{lemma:exp} and $0<|\mu_{n-f}^{k_m}(a)-\mu_{n-f}^{k_m}(b)|<f$, since $F_n=\{n-f+1,\dots,n\}$. However, we can choose indices $a,b\in F_n$ such that $f\leq|\mu_{n-f}(a)-\mu_{n-f}(b)|$, which is a contradiction. If $f$ is even, for example, pick $a,b\in F_n$ such that $\mu_{n-f}(b)=\mu_{n-f}(a)+\tfrac{f}{2}\Big(\tfrac{n-f}{f}\Big)$. If $f$ is odd, pick $a,b\in F_n$ such that $\mu_{n-f}(b)=\mu_{n-f}(a)+\tfrac{f\pm 1}{2}\Big(\tfrac{n-f}{f}\Big)$. Then, $k_m\in\{1,\dots,n-f\}\setminus F_{n-f}$, and this set has exactly $n-2f$ elements. Now, given a quandle of cyclic type of order $n$ with $f$ fixed points such that $n=cf$, where $c>3$, we know $2\in\{1,\dots,n-f\}\setminus F_{n-f}$. Hence, given $a,b\in F_n$, \[\mu_{n-f}(a)-\mu_{n-f}(b)=\mu_{n-f}^2(a)-\mu_{n-f}^2(b).\] However, we also have, by assertion \textit{3.} in Theorem \ref{thr:thethr}, that \[\mu_{\mu_{n-f}(b)-\mu_{n-f}(a)}(\mu_{n-f}(b))=\mu_n^{\mu_{n-f}(b)-\mu_{n-f}(a)}\mu_{n-f}\mu_n^{-(\mu_{n-f}(b)-\mu_{n-f}(a))}(\mu_{n-f}(b))=\]\[=\mu_n^{\mu_{n-f}(b)-\mu_{n-f}(a)}\mu_{n-f}(\mu_{n-f}(a))=\mu_{n-f}^2(a)-\mu_{n-f}(a)+\mu_{n-f}(b)=\mu_{n-f}^2(b)=\mu_{n-f}(\mu_{n-f}(b)),\] implying that the two associate permutations $\mu_{\mu_{n-f}(b)-\mu_{n-f}(a)}$ and $\mu_{n-f}$ (see Proposition \ref{prop:assind} and Corollary \ref{cor:assind}) have the same image at a point that is not a fixed point of these permutations, as $\mu_{n-f}(b)\notin F_{n-f}$. Hence, these permutations must be equal to each other, which is a contradiction, as these permutations are different from each other by Corollary \ref{cor:difper}. Therefore, there are no quandles of cyclic type of order $n$ with $f$ fixed points such that $n=cf$, for $c>3$, and the result follows.
\end{proof}

\subsection{\boldmath Classifying Quandles of Cyclic Type of Order $n$ with $f$ Fixed Points such that $n>2f$}\label{subsect:4}

In this Subsection, we prove Assertion $1.(b)$ in Theorem \ref{thm:main1} and make a few observations regarding $Q_6^2$.

\begin{proof}
(Assertion $1.(b)$ in Theorem \ref{thm:main1}) Immediate from Corollary \ref{cor:order}, Propositions \ref{prop:cyclic1} and \ref{prop:cyclic2} and Corollary \ref{cor:q62}.
\end{proof}

\begin{example}
$Q_6^2$, whose multiplication table is displayed in Table \ref{table:1}, is the only quandle of cyclic type of order $n$ with $f$ fixed points such that $n>2f$, up to isomorphism. In particular, $Q_6^2$ is not a simple quandle (since it admits a non-trivial congruence). Indeed, by Proposition \ref{prop:equivrel}, for $n\geq 2f$, \text{``i is associate to j''} generates an equivalence relation on $Q_6^2$, which is also a congruence relation on this set, as it respects the binary operation of the quandle. In Table \ref{table:2},  we see the quotient of $Q_6^2$ by this congruence relation, which we denote by $\sim$. This quotient is clearly isomorphic to $R_3$, since ``the product'' of any two distinct elements equals the other element.

\begin{center}
\begin{tabular}{|c|c c c|}
 \hline
 \emph{$\overline{\ast}$} & \emph{\{1,3\}} & \emph{\{2,4\}} & \emph{\{5,6\}}\\
 \hline
 \emph{\{1,3\}} & \emph{\{1,3\}} & \emph{\{5,6\}} & \emph{\{2,4\}}\\
 \emph{\{2,4\}} & \emph{\{5,6\}} & \emph{\{2,4\}} & \emph{\{1,3\}}\\
 \emph{\{5,6\}} & \emph{\{2,4\}} & \emph{\{1,3\}} & \emph{\{5,6\}}\\
 \hline
\end{tabular}
\captionof{table}{$Q_6^2/\sim$ multiplication table.}\label{table:2}
\end{center}
\end{example}

\section{Families of Quandles of Cyclic Type of Order $n$ with $f$ Fixed Points in the Range $f+2 \leq n \leq 2f$.}\label{sect:families-range-f+2-n-2f}

\begin{theorem}\label{thm:constructing-quandles}
Let $f$ be an integer strictly greater than $1$ and $n$ a positive integer such that $f+2 \leq n \leq 2f$. Assume further that $(n-f)\, |\, f$. For each $i$ such that $1 \leq i \leq \frac{n}{n-f}$, consider the permutations of $\{ 1, 2, \dots , n   \}$, given by $$ \mu_{(i-1)(n-f)+1}=\mu_{(i-1)(n-f)+1}= \cdots =\mu_{(i-1)(n-f)+1}=(i(n-f)+1\,\,  i(n-f)+2 \,\, \cdots \,\, (i+1)(n-f)) .$$ This sequence of permutations defines a quandle of cyclic type of order $n$ with $f$ fixed points over the set $\{ 1, 2, \dots , n   \}$.
\end{theorem}
\begin{proof} The proof of this Theorem is basically a rearrangement of the argument for the proof of the first statement of Corollary \ref{cor:n=2f}, the existence of a quandle of cyclic type of order $2f$ with $f$ fixed points. We add it here for completeness.

Let $$(i-1)(n-f)+1 \leq j, j' \leq i(n-f) \qquad \text{ and } \qquad (i'-1)(n-f)+1 \leq k \leq i'(n-f) \qquad (i\neq i').$$ Then, $$\mu_j(k) = k+1 \Longrightarrow \mu_{k+1}=\mu_j\mu_k\mu_j^{-1}=\mu_k \qquad \qquad \text{ since $\mu_j$ and $\mu_k$ commute } .$$ Also, $$\mu_j(j') = j' \Longrightarrow \mu_{j'}=\mu_j\mu_{j'}\mu_j^{-1}=\mu_{j'} \qquad \qquad \text{ since $\mu_j$ and $\mu_{j'}$ are equal } .$$ This completes the proof.

\end{proof}

\subsection{Extracting - Adjoining a Common Fixed Point.}\label{sect:ext-adj-common-fixed-point}

\begin{definition}\label{def:common-fixed-point}
Let $Q$ be a quandle of cyclic type with several fixed points. If $g_0 \in Q$ is such that it is a fixed point for any of the permutations of $Q$, $g_0$ is called a {\rm common fixed point of } $Q$.
\end{definition}

\begin{example}
In Table \ref{table:common-fixed-point} we provide the multiplication table of a quandle of cyclic type of order $5$ and $3$ fixed points. The order and number of fixed points of this quandle satisfy $f + 2 \leq n \leq 2f$. Moreover, its permutations are $$ \mu_1 = (1)(2)(3)(4\,\, 5) = \mu_2 = \mu_3 \qquad \qquad   \mu_5 = (1)(4)(5)(2\,\, 3) = \mu_4 .$$ Thus, $1$ is a common fixed point for this quandle.
\end{example}

\begin{center}
\begin{tabular}{|c|c c c c c |}
 \hline
 $\ast$ & 1 & 2 & 3 & 4 & 5 \\
 \hline
 1 & 1 & 1 & 1 & 1 & 1 \\
 2 & 2 & 2 & 2 & 3 & 3 \\
 3 & 3 & 3 & 3 & 2 & 2 \\
 4 & 5 & 5 & 5 & 4 & 4 \\
 5 & 4 & 4 & 4 & 5 & 5 \\
 \hline
\end{tabular}
\captionof{table}{Quandle of cyclic type of order $5$ and $3$ fixed points with a common fixed point: $1$. See \cite{Carter-book}, page $176$.}\label{table:common-fixed-point}
\end{center}

The next two Theorems show us when we can extract a common fixed point (Theorem \ref{thm:extracting-common-fixed-point}) and when we can adjoin a common fixed point (Theorem \ref{thm:adjoining-common-fixed-point}). We repeat their statements here for completeness. Once Theorem \ref{thm:adjoining-common-fixed-point} is proved, combining it with Theorem \ref{thm:constructing-quandles} and iterating the procedure, provides an infinite sequence of quandles of cyclic type with several fixed points within the range $f+2 \leq n \leq 2f$ where $n$ is the order and $f$ the number of fixed points. This is the content of Corollary \ref{cor:an-infinite-sequence-adjoining-common-fixed-points}.

\begin{theorem}\label{thm:extracting-common-fixed-point-bis}
Suppose $f$ is an integer strictly greater than $2$ and $n$ a positive integer such that $f+2 \leq n \leq 2f$. Consider a quandle of cyclic type of order $n$ and $f$ fixed points over the set $Q = \{ 1, 2, \dots , n  \}$ with sequence of permutations $\mu_i$ with $i\in \{  1, 2, \dots , n  \}$. Assume further $g_0 \in Q$ is a common fixed point of $Q$. Then, the set $Q' = Q \setminus \{  g_0 \}$ along with the sequence of permutations $\mu'_i = \mu_i|_{Q'}$ for each $i\in Q'$ defines a quandle of cyclic type of order $n-1$ with $f-1$ fixed points. We call this the extraction of the common fixed point $g_0$.
\end{theorem}
\begin{proof}
We keep the notation and terminology from the statement. Since $\mu_i(g_0) = g_0$, for each $i\in Q$, then $$\mu_{g_0} = \mu_i\mu_{g_0}\mu_i{-1} \Longleftrightarrow \mu_i\mu_{g_0} = \mu_{g_0}\mu_i ,$$ which amounts to saying that $$\mu_{g_0} = \mu_i^{k_i} \qquad (\text{ for some } i\in Q, \text{ for some } k_i \in \mathbf{Z}) \qquad \qquad \text{ OR } \qquad \qquad C_{g_0} \cap C_i = \emptyset .$$

Assume $$\mu_{g_0} = (g_1\quad g_2\quad \dots \quad g_{n-f}) .$$ Then, $$g_{i+1} = \mu_{g_0}(g_i)  \Longrightarrow \mu_{g_{i+1}} = \mu_{g_0}\mu_{g_i}\mu_{g_0}^{-1} = \mu_{g_i} .$$ So the associate permutations to the permutations corresponding to the elements moved by $\mu_{g_0}$, are all equal to one another.

Consider now the set $Q' = Q \setminus \{  g_0 \}$ along with the sequence of permutations $\mu'_i = \mu_i|_{Q'}$ for each $i\in Q'$. For each $i, j \in Q'$, we have $$\mu_{\mu_i(j)} = \mu_i\mu_j\mu_i^{-1} \Longleftrightarrow \mu'_{{\mu'_i}(j)} = \mu'_i\mu'_j{\mu'_i}^{-1} ,$$ which completes the proof.

\end{proof}

\begin{theorem}\label{thm:adjoining-common-fixed-point-bis}
Let $n$ be an integer greater than $2$. Let  be the underlying set of a quandle whose permutations are denoted $\mu_i$, for each $i\in Q$. Let $g_0\notin Q$ and consider the set $Q' = Q \cup \{  g_0 \}$. Suppose there is a permutation, $\mu$, of the elements of $Q$, such that $\mu\mu_i = \mu_i\mu$, for each $i\in Q$. Then, $Q'$ along with the permutations $$\mu'_i = (g_0)\mu_i \qquad \text{ for each } i\in Q \qquad \qquad \text{ and } \qquad \qquad \mu'_{g_0} = (g_0)\mu$$ is a quandle with a common fixed point, $g_0$.
\end{theorem}
\begin{proof}
For each $i, i'\in Q$, $\mu_{\mu_i(i')} = \mu_i\mu_{i'}\mu_i^{-1} \Longleftrightarrow \mu'_{\mu'_i(i')} = \mu'_i\mu'_{i'}{\mu'_i}^{-1}$ and $\mu\mu_i = \mu_i\mu \Longleftrightarrow \mu'\mu'_i = \mu'_i\mu'$. This completes the proof.
\end{proof}



\section{Further Research}\label{sect:further-research}

In this article we looked into the classification of quandles of cyclic type of order $n$ with $f$ fixed points. We realize that these quandles split into three sorts according to the ranges their $(n, f)$'s lie in. If $n>2f$, then these quandles are connected. As a matter of fact, there is only one such quandle which occurs for $n=6$ and $f=2$; it is the octahedron quandle. For each integer $f>2$, there is exactly one such quandle of order $n=2f$ and it is not connected. Finally, in the range $2 < f + 1 < n < 2f$, such quandles are not connected and there seem to be plenty of them.

With the techniques developed in this article, we plan on looking into the classification of other families of quandles like those with constant profile with $f$ fixed points and two non-singular cycles, to begin with. We also plan on taking a fresh look at quandles of cyclic type i.e., when $f=1$.


\end{document}